\documentclass{amsart}

\usepackage{mathptmx}      
\usepackage{graphics}
\usepackage{amstext,amsfonts,amsmath}
\usepackage{amssymb,amscd}
\usepackage{latexsym}
\usepackage{graphicx}

\newtheorem{theorem}{Theorem}

\newtheorem{proposition}[theorem]{Proposition}

\theoremstyle{definition}

\newtheorem{remark}[theorem]{Remark}

\def\k{\text{\it  K}}

\def\Q{\mathbb{Q}}
\def\R{\mathbb{R}}

\def\N{\mathbb{N}}
\def\Z{\mathbb{Z}}

\def\C{\mathbb{C}}
\def\O{\text{\it O}}

\begin{document}


\title{Integral points on plane curves\\ and\\  Plane Jacobian Conjecture over  number fields}
\author{Nguyen Van Chau}

\address{Institute of Mathematics, Vietnam Academy of Science and Technology, 18 Hoang Quoc Viet, 10307 Hanoi, Vietnam.}
\email{nvchau@math.ac.vn}
\thanks{The author was partially supported by Vietnam National Foundation for Science and Technology
Development (NAFOSTED) grant 101.04-2017.12, and Vietnam Institute for Advanced Study in Mathematics (VIASM)}

\date{}

\begin{abstract}  Let $\k$ be a number field and $\O_\k$  the ring of integers of $\k$. In the spirit of Siegel's theorem on integral points on affine algebraic curves, the plane Jacobian conjecture over $\k$ is equivalent to the following statement: if $P,Q\in \O_\k[x,y]$ and $P_xQ_y-P_yQ_x\equiv 1$, then the  curve $P=0$ has more than one integral point.
\end{abstract}

\maketitle
\markboth{Nguyen Van Chau}{Plane Jacobian conjecture over number fields}

\section{introduction} 
The Jacobian conjecture (see \cite{Keller,Bass, EssenBook}) asserts that for any field $\k$  of characteristic zero every polynomial map $F=(F_1,F_2,\dots,F_n)\in \k[X]^n$ with $ \det DF\equiv 1$ is invertible over $\k$, i.e. $F$ has an inverse in $\k[X]^n$, where $X=(X_1,X_2,\dots, X_n)$. It is known that, for the case of number fields, 
 this conjecture is  equivalent to the  statement:  if $F\in \O_\k[X]^n$ and $\det DF\equiv 1$, then the Diophantine equation system $$F_i(X)=0,\; i=1,2,\dots, n,$$  has an integral solution, i.e. a solution in $\O_\k^n$. Here, $\O_\k$ denotes the ring of  integers of $\k$. This equivalence is a consequence of the deep result due to McKenna and Van de Dries (Theorem A, \cite{McKenna}), which says  that,  for every affine  domain $R$ but not a field, surjective polynomial maps of $R^n$  are automorphisms. 

In this article we would like to note that the plane Jacobian conjecture over a number field can be viewed as  a problem on the numbers of integral points on plane algebraic curves.  
\begin{theorem} Let $\k$ be a number field and $\O_\k$ the ring of integers of $\k$. The following statements are equivalent:
\begin{enumerate} 
\item[i)]
 Every polynomial map $F=(P,Q)\in \k[x,y]^2$ with $\det DF\equiv 1$ is invertible over $\k$.
\item[ii)] If $F=(P,Q)\in \O_\k[x,y]^2$ and $\det DF\equiv 1$
, then the  curve $P=0$ has more than one integral point.
\end{enumerate}
\end{theorem}
In view of such equivalences it is worth to consider the plane Jacobian conjecture over number fields from viewpoint of  Diophantine  geometry of plane algebraic curves. 

The proof of Theorem 1, which will be presented in the two next sections, is based on the celebrated Siegel's theorem \cite{Siegel} on integral points on affine algebraic curves.
 The last section is devoted to a short discussion on the thinness of the image set 
$F(O_\k^2)$ 
of a possible counterexample $F\in \O_\k[x,y]^2$ to the Jacobian conjecture.
 
\section{Fiber with infinitely many integral points}
From now on,  $\k$ is a number field and $\O_\k$ the ring of integers of $\k$. By an integral point we mean a point with coordinates in $\O_\k$. 

Let $h\in \C[x,y]$ be a primitive polynomial. Recall, the {\it exceptional value set} $E_h$ of $h$ is the smallest subset of $\C$ such that  such that the restriction $h: \C^2\setminus h^{-1}(E_h)\longrightarrow \C\setminus E_h$ determines a locally trivial fibration. $E_h$ is a finite set, by  a result of R. Thom, and the fibers of $h$ over values outside $E_h$ are diffeomorphic to same a  connected Riemann surface  with genus $g_h$ and $n_h$ punctures. 

\begin{theorem} Suppose $F=(P,Q)\in \O_\k[x,y]^2$ with $\det DF\equiv 1$. If for a value $k\in \O_\k\setminus E_P$ the fiber $P=k$ has infinitely many integral points, then $F$ is invertible over $\k$.
\end{theorem}

\begin{remark}
i) For the case $\k=\Q$, Theorem 2 can be reduced from Bodin's classification \cite{Bodin} for primitive polynomials $h\in \Q[x,y]$ such that for a value $k\in \Q\setminus E_h$ the fiber $h=k$ has infinitely many points in the lattice $\Z^2$. Indeed, by this classification, if such a polynomial $h$ has no singular point, then there is a polynomial automorphism $\Phi$ of $\Q^2$ such that $h\circ\Phi(x,y)=x$. 

ii) For  the case  $\k=\Q$ or an imaginary quadratic field, 
Theorem 2 is true without the condition   $k\in \O_\k\setminus E_P$. In fact, if $F=(P,Q)\in \O_\k[x,y]^2$ with $\det DF\equiv 1$ is not invertible, then the numbers of integral points on the fibers $P=k$, $k\in \O_\k$, are uniformly bounded  by a constant depending on $F$ and $\k$(Lemma 2, \cite{Chau2006}). 
\end{remark}

To prove Theorem 2 we will use the following result due to Druzkowski \cite{Druz}, which is reformulated in a convenience statement.

\begin{theorem}[see Main Theorem, \cite{Druz}] Suppose $F=(P,Q)\in \C[x,y]^2$ with $\det DF\equiv c\neq 0$. Then, $n_P\neq 2$ and $F$ is invertible over $\C$ if and only if $n_P=1$. 
\end{theorem}


\begin{proof}[{\rm \bf Proof of Theorem 2}] Siegel's theorem \cite{Siegel} asserts that every irreducible affine algebraic curve, defined over $\k$ and having infinitely many  integral points, is a rational curve with one or two unique irreducible branches at infinity. Hence, by the assumptions we have $g_P=0$ and $n_P\in \{1;2\}$. The desired conclusion now follows from Theorem 4.
\end{proof}

\section{Proof of Theorem 1}
First, let us recall the definition of the  so-called  clearing map. Let $\Phi\in \k[X]^n$, $\Phi(X)=X+\Phi_2(X)+\cdots+\Phi_m(X),$
where $\Phi_i(X)$ are homogeneous of degree $i$ and $X=(X_1,X_2,\dots, X_n)$.  The {\it clearing map } $^T\Phi$ associated to $\Phi$ and a new variable $T$ is the map defined by 
$$^T\Phi(X):=X+T\Phi_2(X)+\dots+T^{m-1}\Phi_m(X)=\frac{1}{T}\Phi(TX)$$
(see, for example, p.p. 8--9, \cite{EssenBook}). The `` operation" $T\mapsto \; ^T\Phi(X)$ has the following useful elementary properties: 
\begin{enumerate}
\item[a)] For $0\neq r\in \k$, $\det D\Phi=\det D(^r\Phi)$  and $\Phi$ is invertible if and only if so is $^r\Phi$; 

\item[b)] We can choose  $0\neq r\in\O_\k$ such that $r\Phi_i(X)\in \O_\k[X]^n$ for all $i=2,\cdots, m$. For such $r$ we have $^r\Phi\in \O_\k[X]^n$; 

\item[c)] If $\Phi\in \O_\k[X]^n$ and for an $1$-dimension subspace $L$ of $\k^n$ the curve $\Phi^{-1}(L)$ has only a finitely many integral points, we can choose $0\neq r\in \O_\k$ such that $\Phi^{-1}(L)$ does not contain any point of 
$r\O_\k^2\setminus\{(0,0,...,0)\}$. For such $r$ the origin $(0,0,..,0)$ is the unique integral point on the curve $(^r\Phi)^{-1}(L)$, since $^r\Phi(X)=\frac{1}{r}\Phi(rX)$.
\end{enumerate}

Now, we consider the statements (i) and (ii) in Theorem 1 and the statements:
\begin{enumerate}
\item[i')]  Every map $F=(P,Q)\in \O_\k[x,y]^2$ with $\det DF\equiv 1$ is invertible over $\O_\k$;
\item[ii')] If 
$F=(P,Q)\in \O_\k[x,y]^2$ and $\det DF\equiv 1$
, then for every $k\in \O_\k$ the curve $P=k$ has infinitely many integral points.
\end{enumerate}
 
The properties  (a) and (b) in above ensure that to prove or to disprove the Jacobian conjecture over $\k$ it suffices to consider only the polynomial maps with coefficients in $\O_\k$. In particular, we have 
``(i) $\Leftrightarrow $(i')". Furthermore, by (c) we can easy verify ``(ii) $\Leftrightarrow $(ii')". Moreover, the implication ``(i')$\Rightarrow$(ii')" is evident, while the reverse ``(ii')$\Rightarrow $(i')" follows from Theorem 2. Thus, we get ``(i)$\Leftrightarrow$(ii)".

\section{ The image set $F(\O_\k^2)$}
In connection with Theorem 1 and Theorem 2 let us mention here  some relevant facts on the thinness of the  image set 
$F(\O_\k^2)$ 
 of a possible counterexample $F=(P,Q)\in \O_\k[x,y]^2$ with $\det DF\equiv 1$ to the Jacobian conjecture.

Suppose $F=(P,Q)$ is such a counterexample. Then, the geometric degree $\deg_{geo.}F:=[\k(x,y):\k(P,Q)]>1$ and the restriction $$F: \C^2\setminus F^{-1} (E_F)\longrightarrow \C^2\setminus E_F$$
determines an unbranched covering of $\C^2$ with degree $\deg_{geo.}F$, where $E_F$ is the exceptional value set of $F$, i.e. the smallest subset of $\C^2$ outside which $F$ is locally trivial. Note that $E_F$ is an algebraic curve in $\C^2$.

Since $\deg_{geo.}F>1$,  $F(\O_\k^2)$ 
is a  typical thin set 
 in $\k^2$ in sense of J-P. Serre (see p. 121, \cite{Serre}, for definitions). 
As usual, to measure the thinness of thin sets  we may use the counter function $N_M:\R^+\longrightarrow \N$, associated to each subset $M\subset \k^n$ and defined by 
$$N_M(R):= \text{\rm  the number of points in } M \text{ \rm  with height }\leq R.$$ 
By Schanuel's theorem in \cite{Scha}
$$N_{\O_\k^n}(R)\sim cR^{nd} \text{ as } R\rightarrow +\infty\eqno(1)$$
with a number $c>0$, where  $d:=[\k:\Q]$ - the degree of the extension $ \k/\Q$.  By Cohen's theorem in \cite{Cohen} 
for any thin set $M$ in $\k^n$
$$N_M(R)\leq c R^{(n-\frac{1}{2})d} (\log R)^\gamma \text{ as } R\rightarrow +\infty\eqno(2)$$
with some $c>0$ and $\gamma<1$.
 
Consider the case $\k=\Q$ or an imaginary quadratic field.  
In this case, the points of $F(\O_\k^2)$ are closed to $E_F$ and $F(\O_\k^2)$ seems to be very thin. Indeed, as shown in [Proposition 1, \cite{Chau2006}], for each  $(u,v)\in F(\O_\k^2)$ there is a number $t\in \C$ with $\vert t\vert \leq 1$ such that $(u,v+t)\in E_F$. On the other words,
$$F(\O_\k^2)\subset \bigcup_{ u\in \O_\k, (u,v)\in E_F}\{ (u,v+t): \vert t\vert\leq 1\}.\eqno(3)$$ 
In comparing with (1) and (2), we can see that the thinness of the image set $F(\O_\k^2)$  is  at most as those of $\O_\k$. In fact,  we have 

\begin{proposition} Let $\k=\Q$  or an imaginary quadratic field. 
If $F=(P,Q)\in \O_\k[x,y]^2$ with $\det DF\equiv 1$ is not invertible, then 
$$N_{F(\O_\k^2)}(R)\leq c R^d \text{ as } R\rightarrow +\infty \eqno(4)$$
with a number $c>0$.
\end{proposition}
 
\begin{proof}  
Let $M(R)$ denote the set of points $(u,v)\in F(\O_\k^2)$ such that $u$ is of height $\leq R$.
By definitions  and (3) we can see
$$N_{F(\O_\k^2)}(R)\leq \#M(R)\leq n_\k\deg E_F N_{\O_\k}(R)$$, where   $\deg E_F$ is the degree of the exceptional value curve $E_F$ of $F$ and $n_\k$ is the maximum  of the numbers of integers of $\k$ on disks of radium $1$.
This together with (1) implies the  asymptotic estimate (4).
\end{proof}
Note that the constant $n_\k$ in the above proof is finite only for $\k=\Q$ or an imaginary quadratic field. 
It would be interesting to know if the uniformly boundedness of the numbers of integral points on the fibers of $P$, shown in Remark 3, is true for any number field $\k$ and for any possible counterexample $F$ over $\O_\k$ to the Jacobian conjecture. If this is the case we would have the asymptotic estimate (4) for any number field.

The Uniform Bound Conjecture, which is a consequence of the famous Bombieri-Lang conjecture in Arithmetic Geometry (see \cite{Caporaso} for example), says  that the numbers of rational points on algebraic curves over $\k$ of genus $g>1$    are uniformly upper bounded by a constant depended only on $g$ and $\k$.

\begin{proposition} Let $\k$ be a number field. 
 Assuming the Uniform Bound Conjecture, for every polynomial map $F=(P,Q)\in \O_\k[x,y]^2$ with $\det DF\equiv 1$ and $g_P>1$, if  exist,
 $$N_{F(\O_\k^2)}(R)\leq C R^d \text{ as } R\rightarrow +\infty$$
 with a number $C>0$.
 \end{proposition}
 \begin{proof} Let $M:=\{(u,v)\in F(\O_\k^2): u\in \O_\k\setminus E_P\}$ and 
$L_\lambda :=\{(u,v)\in F(\O_\k^2): u=\lambda\}$, $\lambda\in \O_\k$. We can represent
$$F(\O_\k^2)= M\cup (\bigcup_{\lambda\in \O_\k\cap E_P} L_\lambda). $$
Observe that $N_{L_\lambda}(R)\leq N_{\O_k}(R)$ for all $\lambda\in \O_\k$, while by the Uniform Bound Conjecture, there is a constant $C(\k,g_P)$ such that  
$\# L_\lambda\leq C(\k,g_P)$ for all $\lambda\in \O_\k\setminus E_P$. Then, we can see  
$$\begin{aligned} N_{F(\O_\k^2)}(R)&\leq \sum_{\lambda\in E_P\cap O_\k}N_{L_\lambda}(R)+N_M(R)\\
&\leq  \#(E_P\cap O_\k)N_{\O_k}(R)+C(\k,g_P)N_{\O_k}(R)\end{aligned}$$
Hence, by (1)  we get
$$  N_{F(\O_\k^2)}(R)\leq C R^d \text{ as } R \rightarrow +\infty$$
, where $C:=c( \# E_P+C(\k,g_P))<+\infty $ as $\# E_p<+\infty$.

 \end{proof}



\bigskip
\noindent{\it Acknowledgments}: The author wishes to express his thank to the professors Arno  van den Essen and  Ha Huy Vui   for valuable discussions. The author would like to thank  the Vietnam Institute for Advanced Study in Mathematics for their helps.
\bibliographystyle{amsplain}

\end{document}